\documentclass[12pt]{article}
\usepackage{amsmath}
\usepackage{amssymb}
\usepackage{enumerate}
\usepackage{geometry}
\usepackage{times}
\usepackage{verbatim}
\usepackage{tabls}
\usepackage{url}
\usepackage{amsthm}
\usepackage[boxed]{algorithm}
\usepackage{algorithmic}
\usepackage[utf8]{inputenc}
\usepackage[english]{babel}
\usepackage{graphicx}

\theoremstyle{definition}
\newtheorem{theorem}{Theorem}
\newtheorem{lemma}[theorem]{Lemma}
\newtheorem{corollary}[theorem]{Corollary}
\newtheorem{proposition}[theorem]{Proposition}

\title{On the first $k$ moments of the random count of a pattern in a
  multi-states sequence generated by a Markov
  source} 

\date{}

\author{G. Nuel$^{1,2}$\\
\begin{minipage}{0.8\textwidth}
\vspace*{1cm}
{\small $^1$ CNRS, French National Center for Scientific Research, Paris, France}\\
{\small $^2$ Department of Applied Mathematics, University Paris Descartes, France}\\
\end{minipage}
}

\begin{document}

\maketitle

\begin{abstract}
In this paper, we develop an explicit formula allowing to
    compute the first $k$ moments of the random count of a pattern in
    a multi-states sequence generated by a Markov source. We derive
    efficient algorithms allowing to deal both with low or high
    complexity patterns and either homogeneous or heterogenous Markov
    models. We then apply these results to the distribution of DNA
    patterns in genomic sequences where we show that moment-based
    developments (namely: Edgeworth's expansion and Gram-Charlier type
    B series) allow to improve the reliability of common asymptotic
    approximations like Gaussian or Poisson approximations.
\end{abstract}

\section{Introduction}

The distribution of pattern counts in random sequence generated by
Markov source have many applications in a wide range of fields
including: reliability, insurance, communication systems, pattern
matching, or bioinformatics. In this particular field, a common
application is the statistical detection of pattern of interest in
biological sequences like DNA or proteins. Such approaches have
successfully led both to the confirmation of known biological signals
(PROSITE signatures, CHI motifs , etc.) as well as the identification
of new functional patterns (regulatory motifs in upstream regions,
binding sites, etc.). Here follows a short selection of such work:
\cite{KBC92,Hel98,BJVU98,KBSG99,BFWCG00,Spouge02,HKB02,Spouge04}.

From the statistical point of view, studying the distribution of the
random count of a pattern (simple or complex) in a multi-states Markov
chain is a difficult problem. A great deal of efforts have been spent
on this problem in the last fifty years with many concurrent
approaches and we give here only few references (see
\cite{reignier,Lot05,nuel} for more comprehensive reviews). Exact
methods are based on a wide range of techniques like Markov chain
embedding, moment generating functions, combinatorial methods, or
exponential families
\cite{Fu96,StP97,Ant01,Cha05,boeva,Nue06,StS07,boeva2}. There is also
a wide range of asymptotic approximations, the most popular among them
being: Gaussian approximations \cite{PBM89,Cow91,KlB97,Pru95}, Poisson
approximations \cite{God91,GGSS95,ReS99,Erh00} and Large deviations
approximations \cite{DRV01,Nue04}.

More recently, the connexion between this problem and the pattern
matching theory have been pointed out by several authors
\cite{nicodeme,stefanov,lladser,Nue08,ribeca}. Thanks to these
approaches, it is now possible to obtain an optimal Markov chain
embedding of any pattern problem through minimal Deterministic Finite
state Automata (DFA). In this paper, we want to apply this technique
to the exact computation of the first $k$ moments of a pattern count
in a random sequence generated by a Markov source. Our aim is to
provide efficient algorithms to perform these computations both for
low and high complexity patterns and either considering homogeneous
Markov model or heterogeneous ones.

The paper is organized as follow. In a first part, we recall the
principles of optimal Markov chain embedding through DFA. We then
derive from the moment-generating function of the random pattern count
a new expression for its first $k$ moments, and introduce three
different algorithms to compute it. The relative complexity of these
algorithms in respect with previous approaches are then
discussed. Finally, we apply Edgeworth's expansion and Gram-Charlier
type B series techniques to obtain near Gaussian or near Poisson
approximations and show how this allows to improve the reliability of
classical asymptotic approximations with a modest additional cost.

\section{DFA and optimal Markov chain embedding}

\subsection{Sequence model}

Let $(X_i)_{1 \leqslant i \leqslant \ell}$ be a order $d\geqslant 0$
Markov chain over the cardinal $s \geqslant 2$ alphabet
$\mathcal{A}$. For all $1 \leqslant i \leqslant j \leqslant \ell$, we
denote by $X_i^j\stackrel{\text{def}}{=}X_i\ldots X_j$ the subsequence
between positions $i$ and $j$. For all
$a_1^d\stackrel{\text{def}}{=}a_1\ldots a_d \in \mathcal{A}^d$, $b \in
\mathcal{A}$, and $1 \leqslant i \leqslant \ell-d$, let us denote by
$\nu\left(a_1^d\right)\stackrel{\text{def}}{=}\mathbb{P}\left(X_1^d=a_1^d\right)$
the starting distribution and by
$\pi_{i+d}(a_1^{d},b)\stackrel{\text{def}}{=}\mathbb{P}(X_{i+d}=b |
X_i^{i+d-1}=a_1^d)$ the transition probability towards $X_{i+d}$.

\subsection{Pattern count}

Let $\mathcal{W}$ be a finite set of words (for simplification purpose, we assume that
  $\mathcal{W}$ contains no word of length smaller or equal to
  $d$) over
$\mathcal{A}$. We consider the random number $N$ of matching position of
$\mathcal{W}$ in $X_1^\ell$ defined by
\begin{equation}
N
\stackrel{\text{def}}{=} \sum_{i=1}^\ell \mathbb{I}_{\{\mathcal{W} \cap
\mathcal{S}(X_1^i) \neq \emptyset\}}
\end{equation}
where $\mathcal{S}(X_1^i)$ is the set of all the suffixes of $X_1^i$
and where $\mathbb{I}_A$ is the indicatrix function of event $A$.

\subsection{DFA}

As suggested in
\cite{nicodeme,stefanov,lladser,Nue08}, we perform a
optimal Markov chain embedding of the problem through a DFA. We use
here the notations of \cite{Nue08}. Let
$(\mathcal{A},\mathcal{Q},\sigma,\mathcal{F},\delta)$ be a
\emph{minimal} DFA that recognize the language
$\mathcal{A}^*\mathcal{W}$ ($\mathcal{A}^*$ denote the
set of all -- possibly empty -- texts over $\mathcal{A}$) of all texts  over
$\mathcal{A}$ ending with an occurrence of
$\mathcal{W}$. $\mathcal{Q}$ is a finite state space, $\sigma \in
\mathcal{Q}$ is the starting state, $\mathcal{F} \subset \mathcal{Q}$
is the subset of final states, and $\delta: \mathcal{Q} \times
\mathcal{A} \rightarrow \mathcal{Q}$ is the transition function. We
recursively extend the definition of $\delta$ over $\mathcal{Q} \times
\mathcal{A^*}$ thanks to the relation
$\delta(p,aw)\stackrel{\text{def}}{=}\delta(\delta(p,a),w)$ for all $p
\in \mathcal{Q},a\in\mathcal{A},w \in \mathcal{A}^*$. We additionally
suppose that this automaton is non $d$-ambiguous (a DFA
  having this property is also called a $d$-th order DFA in \cite{lladser}) which means that for all $q \in \mathcal{Q}$,
$\delta^{-d}(p)\stackrel{\text{def}}{=}\left\{a_1^d \in
  \mathcal{A}_1^d, \exists p \in \mathcal{Q},
  \delta\left(p,a_1^d\right)=q\right\}$ is either a singleton, or the
empty set. When the notation is not ambiguous, $\delta^{-d}(p)$ may
also denotes its unique element (singleton case).

\subsection{Markov chain embedding}

\begin{theorem}
  We consider the random sequence over $\mathcal{Q}$ defined by
  $\widetilde{X}_0\stackrel{\text{def}}{=}\sigma$ and
  $\widetilde{X}_i\stackrel{\text{def}}{=}\delta(\widetilde{X}_{i-1},X_i)$
  $\forall i,1 \leqslant i \leqslant \ell$. Then $(\widetilde{X}_i)_{i
    \geqslant d}$ is a heterogeneous order 1 Markov chain over
  $\mathcal{Q}'\stackrel{\text{def}}{=}\delta(s,\mathcal{A}^d\mathcal{A}^*)$
   such as, for all
  $p,q \in \mathcal{Q}'$ and $1 \leqslant i\leqslant \ell-d$ the
  starting distribution $\mu_d(p)\stackrel{\text{def}}{=}\mathbb{P}\left(
    \widetilde{X}_d=p \right)$ and the transition matrix
  $T_{i+d}(p,q)\stackrel{\text{def}}{=}\mathbb{P}\left(
    \widetilde{X}_{i+d}=q | \widetilde{X}_{i+d-1}=p \right)$ are given
  by:
\begin{equation}
\mu_d(p)=\left\{
\begin{array}{ll}
  \nu\left(\delta^{-d}(p) \right) & \text{if $\delta^{-d}(p)\neq\emptyset$}\\
  0 & \text{else}\\
\end{array}
\right.;
\end{equation}
\begin{equation}
T_{i+d}(p,q)=\left\{
\begin{array}{ll}
  \pi_{i+d}\left(\delta^{-d}(p),b\right)
  & \text{if $\exists b \in \mathcal{A},\delta(p,b)=q$}\\
  0 & \text{else}
\end{array}
\right..
\end{equation}
\end{theorem}
\begin{proof}
  The result is immediate considering the properties of the DFA. See
  \cite{lladser} or \cite{Nue08} for more details.
\end{proof}

\subsection{Moment generating function}

\begin{corollary}\label{cor:mgf}
  The moment generating function $f(y)$ of $N$ is given by:
\begin{equation}\label{eq:mgf_hetero}
f(y)\stackrel{\text{def}}{=}\sum_{n=0}^{+\infty} \mathbb{P}\left(N=n \right)y^n = \mu_d
\left(\prod_{i=1}^{\ell-d} \left(P_{i+d} + y Q_{i+d}\right)\right)
\mathbf{1}
\end{equation}
where $\mathbf{1}$ is a column vector of ones (in the same manner, we denote by $\mathbf{0}$ is a column vector of zeros) and where, for all $1
\leqslant i \leqslant \ell-d$, $T_{i+d}=P_{i+d}+Q_{i+d}$ with
$P_{i+d}(p,q)\stackrel{\text{def}}{=}\mathbb{I}_{q\notin\mathcal{F}}T_{i+d}(p,q)$
and
$Q_{i+d}(p,q)\stackrel{\text{def}}{=}\mathbb{I}_{q\in\mathcal{F}}T_{i+d}(p,q)$
for all $p,q \in \mathcal{Q}'$.
\end{corollary}
\begin{proof}
  Since $Q_{i+d}$ contains all counting transitions, we keep track of
  the number of occurrence by associating a dummy variable $y$ to
  these transitions. We hence just have to compute the marginal
  distribution at the end of the sequence and sum up the contribution
  of each state.  See
  \cite{nicodeme,stefanov,lladser,Nue08} for more
  details.
\end{proof}

\begin{corollary}\label{cor:mgf_homo}
  In the particular case where $(X_i)_{1 \leqslant i \leqslant \ell}$
  is a homogeneous Markov chain we can drop the indices in $P_{i+d}$
  and $Q_{i+d}$ and Equation (\ref{eq:mgf_hetero}) simplifies
  into
  \begin{equation}\label{eq:mgf_homo}
    f(y) = \mu_d
\left(P + y Q\right)^{\ell-d}
\mathbf{1}.
  \end{equation}
\end{corollary}

Corollary \ref{cor:mgf_homo} can be found explicitely in \cite{lladser} or \cite{ribeca}
but its (however straightforward) generalization to heterogenous model
(Corollary \ref{cor:mgf}) appears to be a new result.

\section{Main result}\label{section:main}

\begin{lemma}
  For all $k \geqslant 0$ we have
  \begin{equation}\label{eq:Fk}
  f^{(k)}(y)=k ! \mu_d \left( \sum_{1 \leqslant i_1 < \ldots < i_k \leqslant \ell-d}
  A_{i,\{i_1,\ldots,i_k\}}(y) \right) \mathbf{1}
  \end{equation}
  where for all $I \subset \mathbb{N}$, $A_{i,I}(y)=P_{i+d} +yQ_{i+d}$
  if $i \notin I$ and $A_{i,I}(y)=Q_{i+d}$ if $i \in I$.
\end{lemma}
\begin{proof}
  The lemma is obvious for $k=0$. We assume now that the lemma is
  true at fixed rank $k$. When derivating Equation (\ref{eq:Fk}),
  the key is then to see that for all $I \subset \mathbb{N}$,
  $A'_{i,I}(y)=\sum_{j \notin I} A_{i,I\cup\{j\}}(y)$. For each
  configuration $I=\{i_1,\ldots,i_{k+1}\}$, it is hence obvious that
  $A_{i,I}(y)$ appears in $A'_{i,I\setminus \{j\}}$ for all $j \in
  I$. This explains the $k+1$ factor which is combined to $k!$ to
  establish the lemma at rank $k+1$.
\end{proof}

\begin{theorem}
  For all $k \geqslant 0$ we have
  \begin{equation}\label{eq:Espk}
  \mathbb{E}\left( \frac{N!}{(N-k)!} \right)=k! [ g(y) ]_{y^k}
  \quad\text{with}\quad 
  g(y)=\mu_d
\left(\prod_{i=1}^{\ell-d} \left(T_{i+d} + y Q_{i+d}\right)\right)
\mathbf{1}
  \end{equation}
  and where $[ g(y) ]_{y^k}$ denotes the coefficient of degree $k$ in
  $g(y)$.
\end{theorem}
\begin{proof}
  By derivating $k$ times the moment generating function $f$ we easily
  get $\mathbb{E}[N!/(N-k)!]=f^{(k)}(1)$. Expanding the expression of
  $g(y)$ at degree $k$ then allows to identify the right term in
  Equation (\ref{eq:Fk}) for $y=1$ thus proving the theorem.
\end{proof}

\begin{corollary}
  In the particular case where $(X_i)_{1 \leqslant i \leqslant \ell}$
  is a homogeneous Markov Equation (\ref{eq:Espk}) simplifies
  into
  \begin{equation}\label{eq:Espk_homo}
  \mathbb{E}\left( \frac{N!}{(N-k)!} \right)=k! [ g(y) ]_{y^k}
  \quad\text{with}\quad 
  g(y)=\mu_d
 \left(T + y Q\right)^{\ell-d}
\mathbf{1}.
  \end{equation}
\end{corollary}

\section{Three algorithms}

\subsection{Full recursion}

For all  $1 \leqslant i \leqslant \ell-d$ we
consider column polynomial vector defined by
  \begin{equation}
  E_i(y)\stackrel{\text{def}}{=}\left( \prod_{j=i}^{\ell-d}  \left(T_{j+d} + y Q_{j+d}\right) \right) \mathbf{1}.
  \end{equation}
  If we denote now by $E_k(i)\stackrel{\text{def}}{=}\left[
    E_i(y)\right]_{y^k}$ its coefficient of degree $k$ for all $k
  \geqslant 0$, then it is clear that we can rewrite the expression of
  $g(y)$ in Equation (\ref{eq:Espk}) as $[g(y)]_{y^k}=\mu_d E_k(1)$.

\begin{proposition}
  We have the following results for all $1 \leqslant i \leqslant \ell-d$:
  \begin{enumerate}[i)]
  \item  $E_0(i)=\mathbf{1}$;
  \item $E_1(\ell-d)=Q_\ell \mathbf{1}$;
  \item if $k \geqslant 1$ and  $(\ell-d-i+1)<k$ then $E_k(i)=\mathbf{0}$;
  \item if $k \geqslant 1$ and $i<\ell-d$ then $E_k(i)=T_{i+d}E_k(i+1)+Q_{i+d} E_{k-1} (i+1)$.
  \end{enumerate}
\end{proposition}
\begin{proof}
  i) It is clear that $E_0(i)=(\prod_{j=1}^{\ell-d} T_{j+d})
  \mathbf{1}$ which is equal to $\mathbf{1}$ since all $T_{j+d}$ are
  stochastic matrices; ii) immediate; iii) the product must contains
  at least $k$ terms to have degree $k$ contribution; iv) is easily
  proved by recurrence using the fact that $E_i(y)=(T_{i+d}+yQ_{i+d})E_{i+1}(y)$.
\end{proof}


\begin{algorithm}
  \begin{algorithmic}
    \REQUIRE{The starting distribution $\mu_d$, matrices $T_{i}$ and
      $Q_{i}$ for all $1 \leqslant i \leqslant \ell-d$, and a $O(k
      \times L)$ workspace to keep the current values of $E_j(i)$
      for $0 \leqslant j \leqslant k$, where $L$ denotes the cardinal
      of $\mathcal{Q}'$.}
    \STATE\COMMENT{Initialization}
    \STATE $E_0(\ell-d)=\mathbf{1}$, $E_1(\ell-d)=Q_\ell \mathbf{1}$, and $E_j(\ell-d)=\mathbf{0}$ for $2 \leqslant j \leqslant k$.
    \STATE\COMMENT{Recursion}
    \FOR{$i=(\ell-d-1)..1$}
    \FOR{$j=k..1$}
    \STATE $E_j(i)=T_{i+d} E_j(i+1) + Q_{i+d} E_{j-1}(i+1)$ 
    \ENDFOR
    \ENDFOR
    \ENSURE{for all $0 \leqslant j \leqslant k$, $[g(y)]_{y^j}=\mu_d
       E_j(1)$}
  \end{algorithmic}
  \caption{Compute the $k$ first terms of $g(y)$ in the most general
    case by performing a full recursion. The workspace complexity is
    $O(k \times L)$ and since all matrix vector product exploit the
    sparse structure of the matrices, the time complexity is $O(\ell
    \times k \times s \times L)$ where $s \times L$ corresponds to the
    maximum number of non zero terms in $T_{i+d}$.}\label{algo1}
\end{algorithm}

\subsection{Direct power computation}\label{section:power}

From now on, we consider the particular case where the Markov model is
homogeneous. According to Equation (\ref{eq:Espk_homo}) the expression
of $g(y)$ in such a case is then simplified into $g(y)=\mu_d
(T+yQ)^{\ell-d} \mathbf{1}$. If we denote by
$M_i(y)\stackrel{\text{def}}{=}[(T+yQ)^{i}]_{y^{0..k}}$ our problem is
then only to compute $M_{\ell-d}(y)$ since $[g(y)]_{y^j}=[\mu_d
M_{\ell-d}(y) \mathbf{1}]_{y^j}$ for all $0 \leqslant j \leqslant
k$. 

\begin{proposition}
We have
\begin{equation}\label{eq:power}
M_{\ell-d}(y)= \prod_{j=0}^J M_{2^j}(y)^{\mathbb{I}_{\{a_j=1\}}} 
\end{equation}
where $\ell-d = a_0 2^0 + a_1 2^1
+\ldots + a_J 2^J$ with $a_j \in \{0,1\}$ for $0 \leqslant j \leqslant J \stackrel{\text{def}}{=}\lfloor \log_2 (\ell-d)
\rfloor$ ($\forall x \in \mathbb{R}$, $\lfloor x \rfloor$
  denotes the largest integer smaller than $x$).
\end{proposition}
\begin{proof}
  Immediate.
\end{proof}

Since we only need to compute the terms of degree smaller than $k$ in
$M_{\ell-d}(y)$ to obtain the first $k$ moments of $N$, we can speed
up the computation by ignoring terms of degree greater than $k$ in
Equation (\ref{eq:power}). We hence obtain Algorithm \ref{algo2} where
$\tau_k[p(y)]$ denotes the truncated polynomial obtained from $p(y)$ by dropping all
terms of degree greater than $k$.

\begin{algorithm}
  \begin{algorithmic}
    \REQUIRE{The starting distribution $\mu_d$, matrices $T$ and
      $Q$, $\ell$, $d$, and $O(k \times L^2 \times J)$ for $M_{2^j}(y)$ for $0\leqslant j \leqslant J$ and a polynomial matrix $M(y)$.}
    \STATE \COMMENT{Preliminary computations}
    \STATE perform the binary decomposition $\ell-d= a_0 2^0 + \ldots a_J 2^J$
    \STATE $M_{2^0}(y)=(P+yQ)^1$
    \FOR{$j=1..J$}
    \STATE $M_{2^{j}}(y) = \tau_k \left[ M_{2^{j-1}}(y)^2 \right]$
    \ENDFOR
    \STATE \COMMENT{Computing $M_{\ell-d}(y)$}
    \STATE $M(y)=M_0(y)$
    \FOR{$j=0..J$}   
    \STATE if $a_j=1$ then $M(y) = \tau_k \left[ M(y) \times M_{2^j}(y) \right]$
    \ENDFOR
    \ENSURE{for all $0 \leqslant j \leqslant k$, $[g(y)]_{y^j}=[\mu_d
       M_{\ell-d}(y) \mathbf{1}]_{y^j}$}
  \end{algorithmic}
  \caption{Compute the $k$ first terms of $g(y)$ in the particular
    case of a homogeneous Markov model through a direct power
    computation. The workspace complexity is $O(k \times L^2 \times
    \log_2 \ell)$ and the time complexity is $O(k^2 \times L^3 \times
    \log_2 \ell)$ ($k^2$ for the polynomial products and $L^3$
    for the matrix products).}\label{algo2}
\end{algorithm}

\subsection{Partial recursion}

In this particular section, we assume that $T$ is an irreducible and
aperiodic matrix and we denote by $\nu$ the magnitude of its second
eigenvalue when we order them by decreasing magnitude.

For all $i\geqslant 0$ we consider the polynomial vector
$F_i(y)\stackrel{\text{def}}{=}(T+yQ)^i \mathbf{1}$, and for all $k
\geqslant 0$ we denote by
$F_k(i)\stackrel{\text{def}}{=}[F_i(y)]_{y^k}$ the term of degree $k$
in $F_i(y)$. By convention, $F_k(i)=\mathbf{0}$ if $i<0$. It is then
possible to rewrite the expression of $g(y)$ in Equation
(\ref{eq:Espk_homo}) as $[g(y)]_{y^k}=\mu_d
F_k(\ell-d)$. Additionnaly, let us finally define recursively the
quantity $D_j^k(i)$ for all $k,i,j \geqslant 0$ by
$D_k^0(i)\stackrel{\text{def}}{=}F_k(i)$ and, if $i \geqslant 1$ and
$j \geqslant 1$,
$D_k^{j}(i)\stackrel{\text{def}}{=}D_k^{j-1}(i)-D_k^{j-1}(i-1)$ so
that
\begin{equation}\label{eq:Dkj}
D_k^j(i)=\sum_{\delta=0}^{j} (-1)^\delta {j \choose \delta} F_k(i-\delta).
\end{equation}

\begin{lemma}\label{lemma:partial}
  We have the following initial conditions:
  \begin{enumerate}[i)]
  \item $\forall i \geqslant 0$, $D_0^0(i)=\mathbf{1}$
  \item $\forall j \geqslant 1$, $D_0^j(i)=(-1)^i {j-1 \choose i}\mathbf{1}$ if $0 \leqslant i \leqslant j-1$, and $D_0^j(i)=\mathbf{0}$ if $i \geqslant j$
  \item $\forall k \geqslant 1$, $D_k^0(0)=\mathbf{0}$, and $D_k^0(i)=T D_k^0(i-1)+Q D_{k-1}^{0}(i-1)$ for $i \geqslant 1$.
  \end{enumerate}
  And for all $k,j,i \geqslant 1$ we have the following recurrence relations:
  \begin{enumerate}[a)]
  \item $D_k^{j}(i)=D_k^{j-1}(i)-D_k^{j-1}(i-1)$
  \item $D_k^j(i)=T D_k^j(i-1)+Q D_{k-1}^{j}(i-1)$
  \end{enumerate}
\end{lemma}
\begin{proof}
  i) It is clear that $D_0^0(i)=T^i \mathbf{1}=\mathbf{1}$ since $T$
  is a stochastic matrix; ii) consequence of i) and Equation
  (\ref{eq:Dkj}); iii) is proved by recurrence; a) is simply the
  definition of $D_k^j(i)$; b) consequence of iii) and of the
  recursive definition of $D_k^j(i)$.
\end{proof}

Lemma \ref{lemma:partial} provides an efficient way to compute all
$D_k^j(i)$ for $0 \leqslant k,j \leqslant K$ and $0\leqslant i
\leqslant \alpha$ (see Algorithm \ref{algo3}). However, these
computations suffer numerical instability in floating point algebra. This phenomenon is emprically studied in section \ref{sec:instability}.

\begin{lemma}\label{lemma:2}
  For all $k \geqslant 1$ we have:
  \begin{enumerate}[i)]
    \item $D_k^k(i)=\sum_{j=k}^i T^{i-j} Q D_{k-1}^k(j-k)$ for $i \geqslant k$;
    \item $\exists \mathbf{C}_k \in \mathbb{R}^L$ such as $D_k^k(i)=\mathbf{C}_k + O(k\nu^{i/k})$ and $D_k^{k+1}(i)=\mathbf{0} + O(k\nu^{i/k})$ for all $i \geqslant 2k$.
  \end{enumerate}
\end{lemma}
\begin{proof}
  i) is a direct application of Lemma \ref{lemma:1}b). For $k=1$, i)
  simply gives $D_1^1(i)=T^{i-1}Q \mathbf{1}$ which proves ii) for
  $k=1$. We assume that ii) is true for some fixed rank $k$ and then decompose
  $D_{k+1}^{k+1}(i)$ into:
  \begin{equation}
    D_{k+1}^{k+1}(i)=\underbrace{T^{i-\alpha}\left(\sum_{j=k+1}^{\alpha} T^{\alpha-j}Q D_{k}^{k+1}(j-k-1)\right)}_A + \underbrace{\sum_{j=\alpha+1}^i T^{i-j} Q D_{k}^{k+1}(j-k-1)}_B
  \end{equation}
  for some $\alpha \geqslant 2k$. Thanks to the stochasticity of $T$,
  $\exists \mathbf{C}^\alpha_{k+1} \in \mathbb{R}^L$ such as
  $A=\mathbf{C}^\alpha_{k+1}+O(\nu^{i-\alpha})$, and since ii) is true
  at rank $k$, $B=\sum_{j=\alpha}^i O(k\nu^{j/k})$.  Elementary
  analysis proves that $\min_\alpha \left\{ \nu^{i-\alpha} +
    \sum_{j=\alpha}^i k\nu^{i'/k} \right\}=O\left((k+1)\nu^{i/(k+1)}\right)$ the
  minimum being obtained for $\alpha=i(k-1)/k$. ii) it then proved at
  rank $k+1$ with $\mathbf{C}_{k+1} = \mathbf{C}^\alpha_{k+1}$ for
  that particular $\alpha$.
\end{proof}

\begin{proposition}
  For all $k \geqslant 1$ and $0 \leqslant j \leqslant k$ and for any
  $i\geqslant \alpha \geqslant 2k$ 
  \begin{equation}\label{eq:Fkibis}
  D_k^j(i)=\sum_{j'=0}^{k-j} {i-\alpha  \choose j'} D_k^{j+j'}(\alpha)+O\left(k{i-\alpha \choose k-j} \nu^{\alpha/k}\right)
  \end{equation}
  and in the particular case where $j=0$ we get:
  \begin{equation}\label{eq:Fki}
  F_k(i)=F_k(\alpha) + \sum_{j'=1}^{k} {i-\alpha  \choose j'} D_k^{j'}(\alpha)+O\left(k {i-\alpha \choose k} \nu^{\alpha/k}\right).
  \end{equation}
\end{proposition}
\begin{proof}
  A simple application of Lemma \ref{lemma:2}ii) proves that
  $D_k^k(i)=D_k^k(\alpha)+O(\nu^{\alpha/k})$ which is exactly Equation
  (\ref{eq:Fkibis}) for $j=k$. We then obtain the result for $j<k$ by
  recurrence and the fact that $D_k^j(i)=D_k^j(\alpha)+\sum_{i'=\alpha+1}^{i} D_k^{j+1}(i')$ and that $\sum_{i'=\alpha+1}^i {i'-\alpha \choose j'}={i-\alpha \choose j'+1}$.

\end{proof}

\begin{algorithm}
  \begin{algorithmic}
    \REQUIRE{
      The matrices $T$ and $Q$, a value $\alpha \geqslant K$, and a $O(K^2 \times L)$ workspace to keep the current value of $D_k^j(i)$ and $D_k^j(i-1)$ for all $0 \leqslant k,j \leqslant K$}
    \FOR{$i=0 .. \alpha$}
    \STATE \COMMENT{Initialization}
    \STATE $D_0^0(i)=\mathbf{1}$ 
    \STATE {\bf for} $j=1 ..K$ {\bf do} $D_0^j(i)=(-1)^i {j-1 \choose i}\mathbf{1}$ if $0 \leqslant i \leqslant j-1$, and $D_0^j(i)=\mathbf{0}$ if $i \geqslant j$ {\bf endfor}
    \STATE {\bf for} $k=1..K$ {\bf do} $D_k^0(i)=\mathbf{0}$ if $i=0$, and
    $D_k^0(i)=T D_k^0(i-1)+Q D_{k-1}^{0}(i-1)$ if $i \geqslant 1$ {\bf endfor}
    \ENDFOR
    \STATE\COMMENT{Recursion}
    \FOR{$k=1 .. K$ and $j=1 .. K$}
    \STATE update $D_k^j(i)$ either with 
    $D_k^{j-1}(i)-D_k^{j-1}(i-1)$ or
    $T D_k^j(i-1)+Q D_{k-1}^{j}(i-1)$
    \ENDFOR
  \end{algorithmic}
  \caption{Compute $D_k^j(\alpha)$ for all $0 \leqslant k,j \leqslant
    K$. The workspace complexity is $O(K^2 \times L)$ and 
    since all
    matrix vector product exploit the sparse structure of the
    matrices, the time complexity is $O(\alpha \times K^2 \times s
    \times L)$.}\label{algo3}
\end{algorithm}


\subsection{Comparison with known methods}

Up to our knowledge, there is no record of method allowing to compute
order $k$ moments of pattern count in heterogeneous Markov
sequences. This work was in fact initially motivated by this
observation. In the homogeneous case however, many interesting
approaches can be found in the literature. In most case, these methods
are limited to the computation of the first two moments, but several
of them can be also used to get arbitrary order moments like with our
method.

One of these approaches consist to consider the bivariate moment
generating function
\begin{equation}
f(y,z)\stackrel{\text{def}}{=}\sum_{ n \geqslant 0, \ell \geqslant d} \mathbb{P}(N_\ell = n)
y^n z ^\ell
\end{equation}
where $N_\ell$ is the random number of pattern occurrences in a
sequence of length $\ell$. Thanks to Equation (\ref{eq:mgf_homo}) it is easy to show that
\begin{equation}
f(y,z)
=z^d \times \mu_d \left( I - z(P+yQ) \right)^{-1} \mathbf{1}
\end{equation}
where $I$ denotes the identity matrix. It is then possible to get
order $k$ moments of $N_\ell$ using the relation:
\begin{equation}
\frac{\partial ^k f}{\partial y^k}(1,z)
=\sum_{\ell \geqslant d}
\mathbb{E}\left(
 \frac{N_\ell !}{(N_\ell - k)!}
  \right)  z^\ell.
\end{equation}
Such interesting approach have been developed by several authors
including \cite{nicodeme} and \cite{lladser}.  In order to apply
this method, one should first use a Computer Algebra System (CAS) to
perform the bivariate polynomial inversion of matrix $I - z(P+yQ)$ to
get $f(y,z)$ thus resulting in a complexity $O(L^3)$ where $L$ is the
number of states in the embedding Markov chain.  One hence needs to
compute the order $k$ partial derivative in $y$ of $f(y,z)$ prior to
to perform (fast) Taylor expansion of the result up to $z^\ell$. The
resulting complexity is $O(\log_2 \ell \times D^3)$ where $D$ is the
degree of the denominator in $\partial ^k f/\partial y^k(1,z)$. Like
in Algorithm \ref{algo2} we get a cubic complexity with $L^3$ for
linear algebra computations, and a logarithmic complexity with $\ell$
thanks to the binary decomposition. However, this method is much more
sophisticated to implement (CAS against simple manipulation of
polynomial matrices) and the $D^3$ term that appears in the Taylor
expansion complexity hide in fact at least a cubic complexity in $k$
which is not easy to handle. Let us note that \cite{nicodeme} also
suggests to obtain asymptotic development of moments by computing only
the local behaviour of the generating function $f(y,z)$ which allows
computation to be performed in faster floating point
arithmetic. However, this approach can not gives the exact moments but
only approximations, and one still require to perform the formal
inversion of an order $L$ bivariate polynomial matrix which is an
expensive step.

More recently, \cite{ribeca} suggested to compute full bulk of the
exact distribution of $N_\ell$ through Equation (\ref{eq:mgf_homo})
using a power method like in Section \ref{section:power} with the
difference that all polynomial products are performed using Fast
Fourier Transform (FFT). The drawback FFT polynomial products is that
the resulting coefficient are known with an absolute precision equal
to the largest one times the relative precision of floating point. As
a consequence, the distribution is well computed only in its center
part. Fortunately, this is precisely the part of the distribution that
matters for moment computations. Using this approach, and a very
careful implementation, one can then compute the full distribution
with a complexity $O(L^3 \times \log_2 \ell \times n_{\text{max}}
\log_2 n_\text{max} )$ where $n_\text{max}$ is the maximum number of
pattern occurrences in the sequence. Once again, the resulting
complexity is likely to be much higher that the one of Algorithm
\ref{algo2} since $k^2$ is usually far smaller than $n_{\text{max}}
\log_2 n_\text{max}$. Moreover, Algorithm \ref{algo2} is again much
easier to implement than this sophisticated FFT approach.

Finally, one should note that both these two known approaches involve
a complexity $O(L^3)$ in time (and at least $O(L^2)$ in memory) which
makes difficult or even impossible to use them for moderate or high
complexity patterns (ex: $L=100$ or $L=1000$). For such patterns,
Algorithm \ref{algo1} appears to be a safe but slow alternative
(linear complexity with sequence length $\ell)$ and Algorithm
\ref{algo3} seems to be a very promising approach since it allows to
handle such complex patterns while retaining a logarithmic complexity
with $\ell$ like in Algorithm \ref{algo2}. Unfortunately, the
numerical instabilities observed in practice with Algorithm
\ref{algo3} need to be investigated further before to trust this
approach.

\section{Application to DNA patterns in genomics}

\subsection{Dataset}\label{section:data}

We consider the a order $d=1$ homogeneous Markov model over
$\mathcal{A}=\{{\tt A},{\tt C},{\tt G},{\tt T}\}$ which transition
matrix estimated over the complete genome of the bacteria {\it
  Escherichia. coli} is given by:
$$
\pi=\left(
  \begin{array}{cccc}
    0.30&   0.21&   0.22&  0.27\\   
    0.23&   0.23&   0.33&  0.22\\   
    0.28&   0.29&   0.23&  0.20\\   
    0.19&   0.28&   0.23&  0.30\\
  \end{array}
\right)
$$
We consider a sequence $X=X_1 \ldots X_\ell$ of length $\ell=400\,000$ and
starting with $X_1={\tt A}$.

\subsection{Some moments}

In this section, we compute the first $k=4$ moments of several DNA
patterns. We then use these moments to compute:
$$
\text{expectation $m=m_1$}, \quad \text{standard deviation $\sigma=\sqrt{m_2}$}
$$
$$
\text{skewness $\gamma_1=m_3/m_2^{3/2}$}, \quad \text{and excess kurtosis $\gamma_2=m_4/m_2^2-3$}
$$
where $m_i \stackrel{\text{def}}{=} \mathbb{E}[(N-m_1)^i]$ is the centered
moment of order $i$. A negative (resp. positive) skewness indicates
that the mass of the distribution is concentrated on the right
(resp. left) side of the expectation. A skewness of zero indicates a
balanced distribution. A negative (resp. positive) excess kurtosis
indicates that the distribution is more flat (resp. more peaked) than
the Gaussian distribution. A Gaussian distribution has a excess
kurtosis of zero.

\begin{table}
  \begin{center}
  \caption{First four moments of several DNA patterns computed through the power algorithm (running time indicated in seconds). The background model is the order $d=1$ homogeneous Markov model defined in section \ref{section:data} and the sequence length is $\ell=400,000$.}\label{table:moments}
    \begin{tabular}{crrrrrr}
      \hline
      \text{Pattern} & L & \text{exp.} &  \text{std. dev.} & \text{skewness} & \text{ekurtosis} & \text{time} \\
      \hline      
      ${\tt GCTGGT}$ & 9 & 70.09 & 8.364 & 0.11910 & 0.01413 & 0.09\\
      ${\tt AGAGAG}$ & 9 & 84.89 & 9.791 & 0.12780 & 0.01903 & 0.09\\
      ${\tt GGGGGG}$ & 9 & 65.91 & 10.260 & 0.20290 & 0.05363 & 0.09\\
      \hline
      ${\tt GCTGGTGG}$ & 11 & 3.782 & 1.945 & 0.51420 & 0.26430 & 0.11\\
      ${\tt GCTGGNGG}$ & 14 & 20.79 & 4.559 & 0.21920 & 0.04801 & 0.11\\
      ${\tt GNTGGNGG}$ & 21 & 79.55 & 9.014 & 0.11570 & 0.01390 & 0.49\\
      ${\tt GNTGNNGG}$ & 28 & 340.1 & 18.680 & 0.05628 & 0.00331 & 1.10\\
      ${\tt GNNGNNGG}$ & 63 & 1508.0 & 42.290 & 0.03283 & 0.00136 & 15.80\\
      \hline      
    \end{tabular}
  \end{center}
\end{table}

On Table \ref{table:moments} we can see the value of these quantities
for several DNA patterns. For the first three simple patterns, we can
see how the additional information off skewness and excess kurtosis
gives us a better description of their distribution. For example, we
know from theory that highly overlapping patterns are distributed
according to compound Poisson approximations. This is exactly why we
observe an increasement of skewness and kurtosis from Pattern ${\tt
  GCTGGT}$ (non-overlapping) to Pattern ${\tt GGGGGG}$ (highly
self-overlapping).

If we consider now the more complex patterns of the second part of
Table \ref{table:moments} we can observe how the running time of
Algorithm \ref{algo2} quickly increases with $L$. This is obviously
not a surprise since we expect a cubic complexity in this parameter
with this approach. One should however note that it is nevertheless
possible to deal with moderately complex patterns like ${\tt
  GNNGNNGG}$ which contains in fact a total of $4^4=256$ simple
patterns. Another interesting observation is that both skewness and
kurtosis get closer to zero when we add more symbol ${\tt N}$ into the
pattern. This is due to the fact that adding more ${\tt N}$ makes the
pattern more frequent (this can be seen with the geometrically
increasing expectation) and that Gaussian approximations for pattern
problem are well known to work better for frequent patterns.

\subsection{Numerical stability of the partial recursion}\label{sec:instability}

\begin{figure}
  \begin{center}
    \includegraphics[clip,trim=0 0 0 0,width=0.7\textwidth,angle=270]{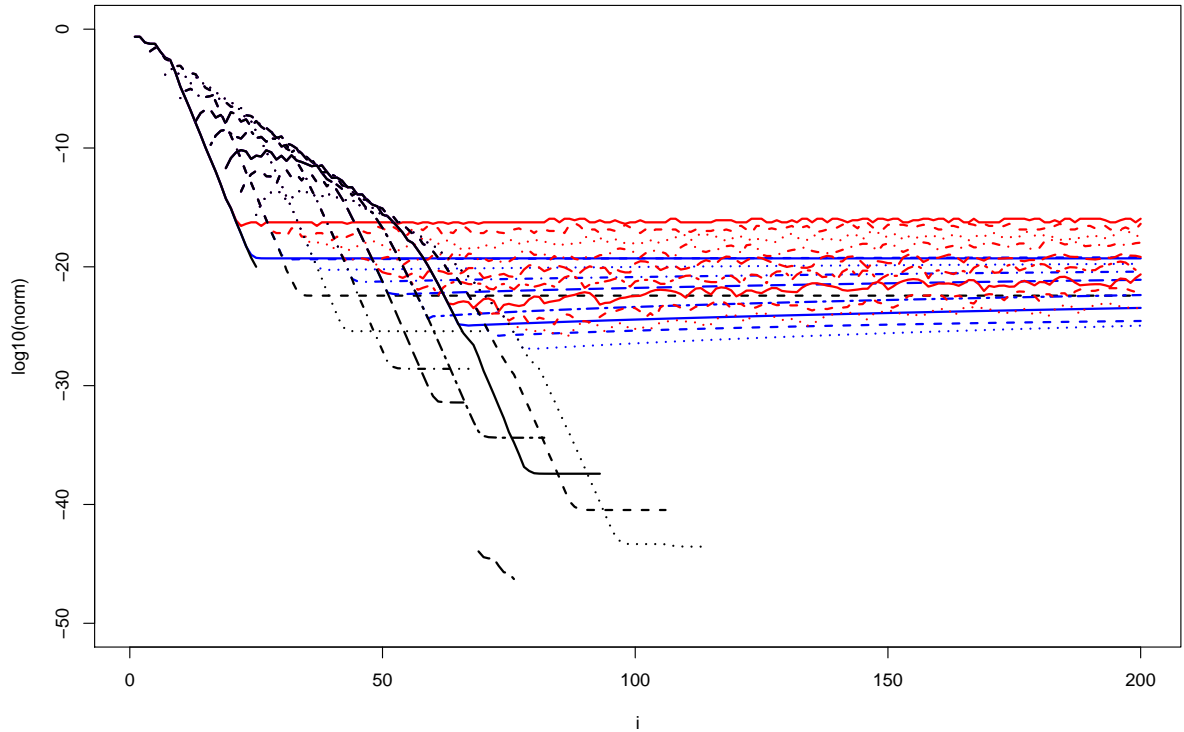}
  \end{center}
  \caption{Plot of $\log_{10} \left|\left|D_{k}^{k+1}(i)
      \right|\right|_\infty$ ($y$-axis) for $1 \leqslant k \leqslant
    9$ (from left to right), and $1 \leqslant i \leqslant 100$
    ($x$-axis) for the pattern $\mathcal{W}={\tt GNTGNNGG}$ over the DNA
    alphabet $\mathcal{A}=\{{\tt A},{\tt C},{\tt G},{\tt T}\}$ (${\tt N}$ symbol meaning ``any letter'') using a
    order $d=1$ Markov model. The curves are obtained through
    Algorithm \ref{algo3} using recurrence relation
    Lemma \label{lemma:1}: a) only (Red curve); b) only (Blue curve);
    a) and b) keeping the $D_k^j(i)$ displaying the smallest norm
    (Black curve). All missing values correspond to
    $\left|\left|D_{k}^{k+1}(i) \right|\right|_\infty=0$.}
  \label{fig:instability}
\end{figure}

On Figure \ref{fig:instability} we study empirically the convergence
of $D_k^{k+1}(i)$ towards $\mathbf{0}$ by computing
$\left|\left|D_{k}^{k+1}(i) \right|\right|_\infty$ for several $k$
through Algorithm \ref{algo3}. We consider here three way of updating
$D_k^j(i)$: by using only through $D_k^{j-1}(i)-D_k^{j-1}(i-1)$ (Red
curve); by using only through $T D_k^j(i-1)+Q D_{k-1}^{j}(i-1)$ (Blue
curve); or by taking the update which displays the smallest
norm (Black curve). If these three alternative approaches give similar
results when $\left|\left|D_{k}^{k+1}(i) \right|\right|_\infty
\geqslant 10^{-15}$ differences start to appear for smaller
values. The differential recurrence relation (Red curve) quickly start
to accumulate machine precision residuals and results in noisy curves
with a slow increasement. When using the matrix recurrence relation
(Blue curve) a similar problem arise, however appearing slightly later
and with far less noise. Surprisingly, the last approach which combine
the two updating methods at each step benefits from a synergistic
effect and displays a far better stability. A similar behaviour have
been observed for a wide range of tested patterns (data not shown).


\subsection{Near Gaussian approximations}

Gaussian approximations for random pattern counts are widely used in the
literature. We want here to push forward this idea by taking advantage
of higher order moments to get near Gaussian approximations. This well
known technique is described in details in Appendix
\ref{sec:edgworth}.

\begin{figure}
  \begin{center}
  \makebox{\includegraphics[clip,trim=0 0 0 0,angle=270,width=1.0\textwidth]{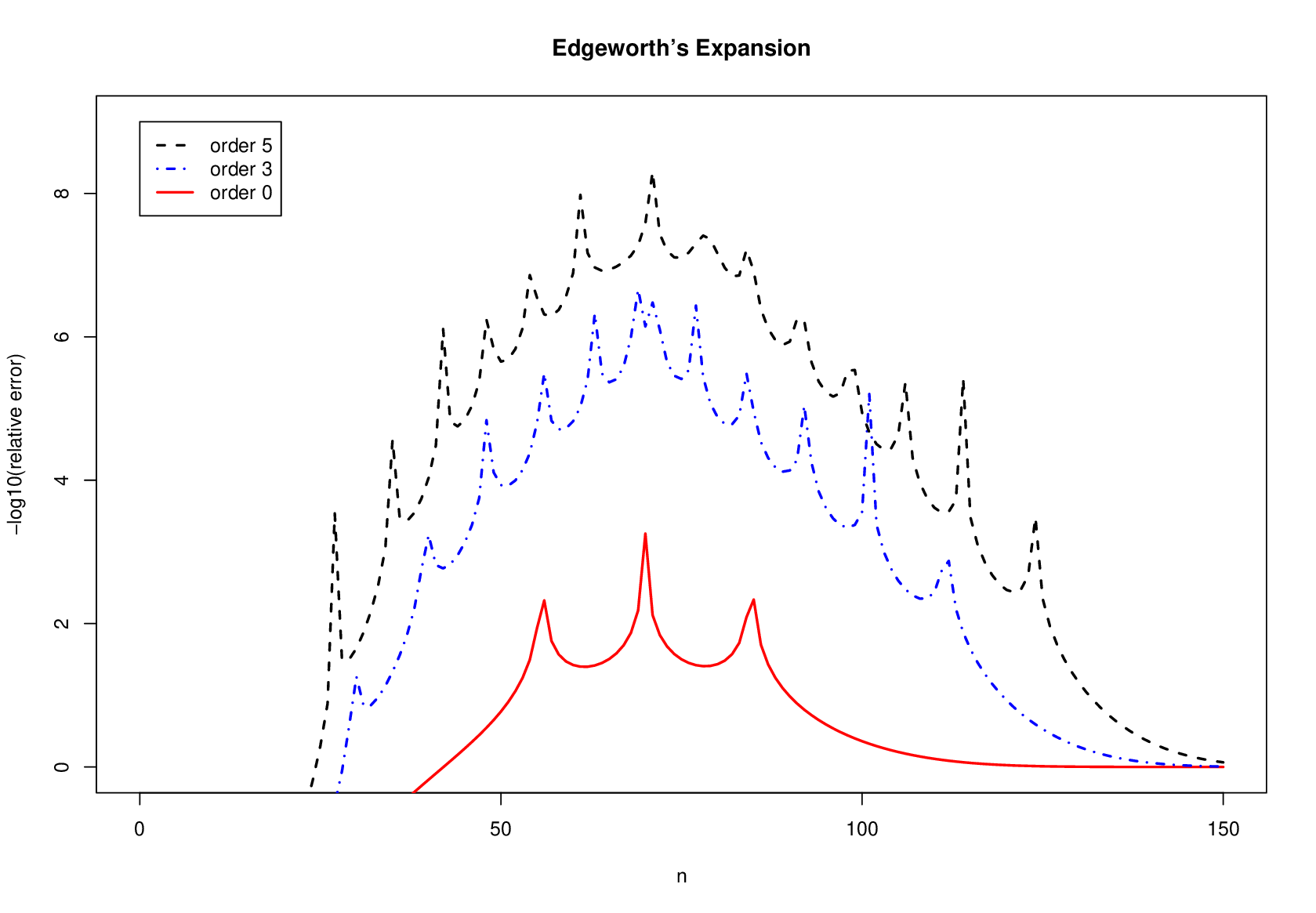}}
  \end{center}
  \caption{Relative error in decimal log scale of Edgeworth's
    expansion of order $s=0$ (Red-solid), order $s=3$
    (Blue-dotdashed), and order $s=5$ (Black-dashed)
    for Pattern ${\tt GCTGGT}$ on a order 1 homogeneous Markov
    model (parameter estimated on the complete genome of {\it
      E. coli}) of length $\ell=400\,000$.}\label{fig:edgeworth1}
\end{figure}

We can see on Figure \ref{fig:edgeworth1} the relative error (in
log-scale) of several Edgeworth's approximations for the distribution
of pattern ${\tt GCTGGT}$. The solid line shows the reliability of
plain Gaussian approximation (which correspond to an order $s=0$
Edgeworth's expansion). Unsurprisingly, this approximation works
better around the expectation ($\mathbb{E}[N]=70.09$ according to
Table \ref{table:moments}) providing two exact digits on the range
$[54;85]$, and one exact digit on the range $[50;92]$. Beyond these
limit, we get too far in the tail distribution to get reliable
results. This behaviour is exactly what we expect from the central
limit theory.

If we consider now order $s=3$ Edgeworth's expansion (that uses
moments up to order $k=5$) depicted with a dotdashed line on Figure
\ref{fig:edgeworth1}, we see a dramatic improvement both on the
accuracy of the approximation (up to 6 exact digits) and on the range
of reliability (at least one exact digit on $[28;118]$). We can even
get a further improvement by considering order $s=5$ expansion (dashed
line) which uses moments up to order $k=7$. In both case however, the
reliability of these approximations decreases dramatically when we get
far enough in the tail distributions.

We observe a very similar behaviour for Pattern ${\tt AGAGAG}$ and
Pattern ${\tt GGGGGG}$ and the corresponding figures are hence not
shown to save space.

Thanks to this work we see that for a modest additional cost
(computing moments up to order $k=5$ or $k=7$ instead of simple first
and second moments), one can dramatically improve the reliability of
Gaussian approximations for pattern problems.

\subsection{Near Poisson approximations}

A very common alternative to Gaussian approximations for random
pattern counts is to turn towards Poisson approximations. These
approximations are known to be quite accurate for non-overlapping
patterns, but also to fail for highly self overlapping patterns for
which compound Poisson approximations are known to perform better. We
want here to evaluation the interest of near Poisson approximations
provided by the Gram-Charlier Type B series described in Appendix
\ref{sec:gram_charlier}.

\begin{figure}
  \begin{center}
  \includegraphics[clip,trim=0 0 0 0,angle=270,width=1.0\textwidth]{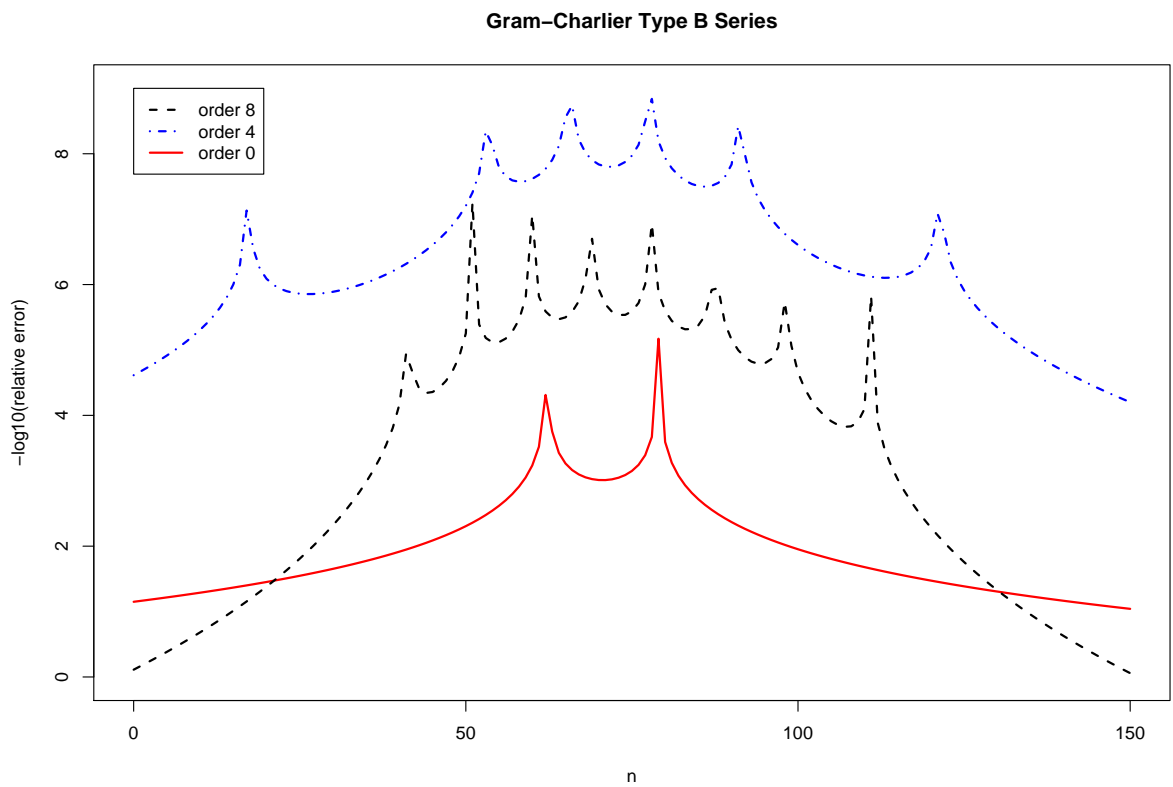}
  \end{center}
  \caption{Relative error in decimal log scale of Gram-Charlier type B
    approximation of order $s=0$ (Red-solid) to order $s=4$ (Blue-dotdashed) to order
    $s=8$ (Black-dashed) for Pattern ${\tt GCTGGT}$ on a order 1 homogeneous
    Markov model (parameter estimated on the complete genome of {\it
      E. coli}) of length $\ell=400\,000$.}\label{fig:poisson1}
\end{figure}

For the non-overlapping pattern ${\tt GCTGGT}$, we can see on Figure
\ref{fig:poisson1} that the plain Poisson approximation (order $s=0$
Gram-Charlier Type B series) gives already very good results with at
least one exact digit on all the distribution, and up to 4 or 5 of
them in the region close to the expectation. This interesting result
is dramatically improved by the order $s=4$ approximations which gives
at least 4 exact digits on all the considered range and more that 8
exact digits around the expectation. Surprisingly, the order $s=8$
approximation is less reliable than the previous one, and gives even
worse results that the plain Poisson approximation in the tail
distributions. This is due to the fact that the coefficients $c_k$
computed according to Equation (\ref{eq:ck}) accumulate large terms
that compensate each other. This is a typical scenario for large
relative errors in floating point arithmetic. One can solve this
problem either by performing computations with an arbitrary number of
digits (usually slow=), or one can explicitly compute the expected
relative error with the current machine-precision and renounce to use
unreliable coefficients.

\begin{figure}
  \begin{center}
  \includegraphics[clip,trim=0 0 0 0,angle=270,width=1.0\textwidth]{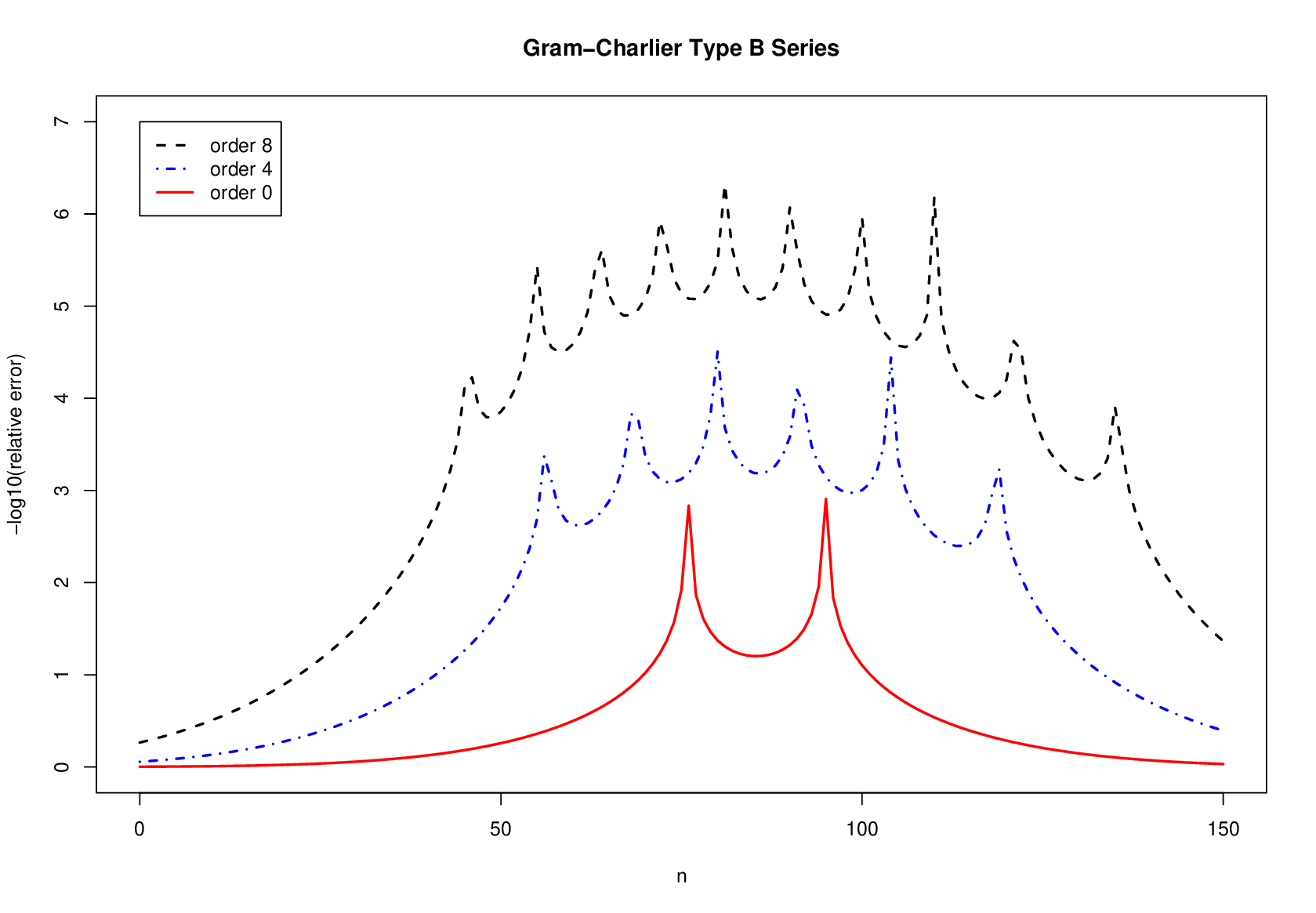}
  \end{center}
  \caption{Relative error in decimal log scale of Gram-Charlier type B
    approximation of order $s=0$ (Red-solid) to order $s=4$ (Blue-dotdashed) to order
    $s=8$ (Black-dashed) for Pattern ${\tt AGAGAG}$ on a order 1 homogeneous
    Markov model (parameter estimated on the complete genome of {\it
      E. coli}) of length $\ell=400\,000$.}\label{fig:poisson2}
\end{figure}

If we consider now the self-overlapping pattern ${\tt AGAGAG}$, we
know from theory that Poisson approximations are not supposed to
perform well. This is the reason why we observe on Figure
\ref{fig:poisson2} that the plain Poisson approximations only works on
a very limited range the distribution (roughly on $[69;103]$). Once
again however, order $s=4$ or $s=8$ Gram-Charlier expansion
dramatically improve the reliability of the approximations getting up
to 6 exact digits close to the expectation and at least one exact
digits on a much wider range (up to $[24;150]$ for order $s=8$). One
should note that in this case, the numerical issue observed for high
order approximations for the previous pattern does not occur. We get a
very similar result for the even more self-overlapping pattern ${\tt
  GGGGGG}$ and the corresponding figure is then omitted to save space.

Like with near Gaussian approximations, we see that near Poisson
approximations can dramatically improve the reliability of Poisson
approximations for a very modest cost (ex: computing moments up to
order $k=4$ or $k=8$).

\section{Conclusion}

In this paper, we have derived from the explicit expression of the mgf
of a pattern random count $N$, a new formula allowing to compute a
arbitrary number $k$ of moments of $N$. We also have introduced three
efficient algorithms to perform this computation. The first one allow
the computation of pattern count moments of arbitrary order in the
framework heterogeneous Markov model which is a completely new result
(up to our knowledge). The second algorithm, suitable for homogeneous
models and low complexity patterns, appear to have a better or similar
complexity to state-of-the art known algorithms but with a far much
simpler implementation. Finally, the third algorithms uses partial
recursions exploiting the sparse structure of the transition matrix to
provide a logarithmic complexity with the sequence length even for
high complexity patterns. This very promising approach however suffers
from numerical instabilities in floating point arithmetic that need to
be further investigated.

One should note that our main result can be easily extended to mixed
moments of several pattern counts. In order to save space, we give
here such as result only for the particular case of two patterns
$\mathcal{W}_1$ and $\mathcal{W}_2$ in a homogeneous model. We assume
that the final states of or DFA could be partitioned into
$\mathcal{F}=\mathcal{F}_1 \cup \mathcal{F}_2$ such as $\mathcal{F}_1$
(resp. $\mathcal{F}_2$) count the number $N_1$ (resp. $N_2$) of
occurrences of $\mathcal{W}_1$ (resp. $\mathcal{W}_2$). This is always
possible by duplicating states. We consider
\begin{equation}
  f(y_1,y_2)\stackrel{\text{def}}{=}
  \sum_{n_1,n_2 \geqslant 0} \mathbb{P}(N_1=n_1,N_2=n_2) y_1^{n_1} y_2^{n_2}
\end{equation}
and we then have $f(y_1,y_2)=\mu_d (P+y_1 Q_1 +y_2 Q_2)^{\ell-d}
\mathbf{1}$. By introducing now
$g(y_1,y_2)\stackrel{\text{def}}{=}\mu_d \left(T + y_1 Q_1+ y_2
  Q_2\right)^{\ell-d}\mathbf{1}$ we get for any $k_1,k_2 \geqslant 0$ that:
\begin{equation}
  \mathbb{E}\left( \frac{N_1!}{(N_1-k_1)!} \times 
\frac{N_2!}{(N_2-k_2)!}
 \right)=k_1!k_2! [ g(y_1,y_2) ]_{y_1^{k_1} y_2^{k_2}}.
\end{equation}

As an application, we have considered the distribution of DNA patterns
in genomic sequences. In this particular framework, we have shown how
order $k=3$ and $k=4$ moments allow to get a better description of the
distribution (with quantities like skewness and excess kurtosis). We
have also considered moment-based approximations namely Edgeworth's
expansion (near Gaussian approximations) and Gram-Charlier Type B
series (near Poisson approximations). For both approximations, we have
seen how the additional information provided by a couple of higher
order moments can dramatically improve the reliability of these common
approximations. As a perspective, it seems to be very promising to
develop near geometric or compound Poisson distribution with
Gram-Charlier Type B series.

\appendix

\begin{center}
    {\bf \Large APPENDIX}
  \end{center}

\section{Moments and cumulants}\label{appendix:cum}

For any random variable $X$ and for any $k \geqslant 0$ we define the
following quantities: $g_k\stackrel{\text{def}}{=}1/k!\mathbb{E}\left[
  X!/(X-k)!\right]$ the coefficient of degree $k$ in the polynomial
$g(y)$ defined in Section \ref{section:main}; $m'_k
\stackrel{\text{def}}{=} \mathbb{E}(X^k)$ the moment of order $k$;
$m_k \stackrel{\text{def}}{=} \mathbb{E}[(N-m'_1)^k]$ the centered
moment of order $k$; and $\kappa_k$ the cumulant of order $k$ defined
by $h(t)\stackrel{\text{def}}{=}\log \mathbb{E}(e^{tN})=\sum_{k
  \geqslant 1} \kappa_k (t^k / k !)$.  Cumulants and moments are
connected through the following formula:
\begin{equation}
\kappa_k = m'_k - \sum_{l=1}^{k-1} {k-1 \choose l-1} \kappa_l m'_{k-l}.
\end{equation}
Using this formula we get: $\kappa_1=\mathbb{E}(X)$ and $\kappa_2 =
m_2=\mathbb{V}(X)$, $\kappa_3=m_3$, and $\kappa_4=m_4-3m_2^2$.  The
skewness $\gamma_1$ and excess kurtosis can be expressed from
cumulants: $\gamma_1 = \kappa_3 / \kappa_2^{3/2}$ and $\gamma_2 =
\kappa_4 / \kappa_2^2$.

\section{Edgeworth's expansion}\label{sec:edgworth}

This is directly taken from \cite{edgeworth} except the explicit
order 5 expansion given in Equation (\ref{eq:edgeworth}) which is a new contribution
(only order 3 explicit expansions seems to be available in the literature).

Let $X$ be a centered random variable ($\mathbb{E}[X]=0$) that admit
finite moments of all orders (we denote by $\sigma^2$ the variance of
$X$), let $\Phi$ defined by
$\Phi(t)\stackrel{\text{def}}{=}\mathbb{E}[e^{iX}]$ (where $i$ denote
the imaginary complex number) be its caracteristic function. Let
$\varphi$ be the caracteristic function of $X/\sigma$, we have
$\varphi(t)=\Phi(t/\sigma)$. The definition of cumulants (see Appendix
\ref{appendix:cum}) then allows to write the expansion:
\begin{equation}
\log \phi(t)=\log \Phi(t/\sigma) \sim \sum_{k=2}^{\infty} \frac{\kappa_k}{\sigma^k k!}(it)^k
\end{equation}
then by denoting $S_k \stackrel{\text{def}}{=} \kappa_k/\sigma^{2k-2}$ we get
\begin{equation}\label{eq:phi_t}
\phi(t) \sim \exp \left\{
\sum_{r=1}^{\infty} \frac{S_{r+2} \sigma^r}{(r+2)!}(it)^{r+2}
\right\}.
\end{equation}

The Fourier transform of expansion (\ref{eq:phi_t}) then gives:
\begin{equation}\label{eq:Edge}
q(x)=Z(x)\left(1+\sum_{s=1}^{\infty} \sigma^s \times \left\{
    \sum_{\{k_m\}_s} H_{s+2r}(x) \prod_{m=1}^{s} \frac{1}{k_m!} \left(
      \frac{S_{m+2}}{(m+2)!}
\right)^{k_m}
\right\} \right)
\end{equation}
where $q(x)\stackrel{\text{def}}{=} \sigma p(\sigma x)$ is the
probability distribution function (pdf) of $X/\sigma$ ($p(x)$ being
the pdf of $X$), where $Z(x)=\exp(-x^2/2)/\sqrt{2\pi}$ is the pdf of a
standard Gaussian variable, where $\{k_m\}_s$ is the set of all
non-negative integer solution of the Diophantine equation
$k_1+2k_2+\ldots+s k_s = s$, $r=k_1+k_2+\ldots+k_s$, and where $H_k(x)$ are
the Hermite polynomials defined recursively by
$H_0(x)\stackrel{\text{def}}{=}1$ and
$H_k(x)\stackrel{\text{def}}{=}xH_{k-1}(x)-H_{k-1}'(x)$ for all $k
\geqslant 1$.

Here are the sets of $\{k_m\}_s$ for $1 \leqslant s \leqslant 5$:
$\{k_m\}_1=\{1\}$, $\{k_m\}_2=\{20,01\}$, $\{k_m\}_3=\{300,110,001\}$,
$\{k_m\}_4=\{4000,2100,0200,1010,0001\}$, and
$\{k_m\}_5=\{50000,31000,12000,$ $20100,01100,10010,00001\}$, and here is
the explicit expression of (\ref{eq:Edge}) up to order $s=5$
(such an explicit expression can be found up to $s=3$
in \cite{BeK95}):
\begin{multline}\label{eq:edgeworth}
  \frac{q(x)}{Z(x)} \simeq
    1 + \sigma \left\{H_3(x) \frac{S_3}{3!} \right\}\\
    +\sigma^2 \left\{ H_4(x) \frac{S_4}{4!} + H_6(x) \frac{S_3^2}{2!3!^2}\right\}
    +\sigma^3 \left\{  H_5(x) \frac{S_5}{5!} + H_7(x) \frac{S_3 S_4}{3!4!} + H_9(x) \frac{S_3^3}{3!^4} \right\}\\
    +\sigma^4 \left\{  H_6(x) \frac{S_6}{6!} + H_8(x) \left(\frac{S_3 S_5}{3!5!} + \frac{S_4^2}{2!4!^2}\right) + H_{10}(x) \frac{S_3^2 S_4}{2!3!^24!} + H_{12}(x) \frac{S_3^4}{4!3!^4}\right\}\\
    +\sigma^5 \left\{  H_7(x) \frac{S_7}{7!} + H_9(x) \left( \frac{S_4 S_5}{4!5!}+\frac{S_3 S_6}{3!6!}\right) + H_{11}(x) \left( \frac{S_3^2 S_5}{2!3!^25!}
+\frac{S_3 S_4^2}{2!3!4!^2}  \right) \right.\\
 \left.  + H_{13}(x) \frac{S_3^3 S_4}{3!^44!}+ H_{15}(x) \frac{S_3^5}{5!3!^5} \right\}
\end{multline} 


\section{Gram-Charlier type B serie for near Poisson distribution}\label{sec:gram_charlier}

This is initially taken from \cite{gram-charlier} but we derive new
recurrence relation that are more adapted to a modern computational
framework than the explicit (and sometimes erroneous) formulas given
in the original article.

Let $\psi (i)\stackrel{\text{def}}{=} e^{-\lambda} \lambda^i / i!$ be
the pdf of a Poisson distribution of parameter $\lambda$, and let
$\Delta$ be the differential operator defined by $\Delta \psi
(i)\stackrel{\text{def}}{=}\psi(i)-\psi(i-1)$. Our objective is to
approximate the pdf $F$ of a discrete non-negative random variable $X$
with
\begin{equation}
F(i)\simeq \sum_{j=0}^{s} c_j \Delta^j \psi(i) 
\end{equation}
In order to do so we use a moment method and find a solution
$(c_0,c_1,\ldots,c_s)$ of $\sum_{j=0}^{s} c_j P_k^{j}(\lambda) =
\mathbb{E}[X^k]$ for all $0 \leqslant k \leqslant s$ with
$P_k^j(\lambda)\stackrel{\text{def}}{=} \sum_{i \geqslant 0}
i^k\Delta^j \psi (i)$ for all $j,k \geqslant 0$.

It is clear that we have $P_0^0(\lambda)=1$, and we have the following
recurrence relation for all $k,j \geqslant 0$:
\begin{equation}
P_{k+1}^0(\lambda)=\lambda\left[
  P_{k}^0(\lambda)+\frac{d P_{k}^{0}}{d\lambda}(\lambda)
\right]
\quad\text{and}\quad
P_k^{j+1}(\lambda)=-\frac{d P_{k}^{j}}{d\lambda}(\lambda).
\end{equation}
We hence get that $c_0=1$ and we derive the following recurrent relation for
$k\geqslant 1$:
$$
 c_k=\frac{1}{P_k^k(\lambda)} \left(
\mathbb{E}[X^k]-\sum_{j=0}^{k-1} c_j P_k^j(\lambda)
\right).
$$
Please note that $P_k^k(\lambda)$ is always a scalar. If we now denote
by $g_k\stackrel{\text{def}}{=}1/k!\mathbb{E}\left[ X!/(X-k)!\right]$
the we can show by recurrence for all $k \geqslant 1$ that we finally have:
\begin{equation}\label{eq:ck}
c_k=-\frac{(k-1)}{k!}g_1^k+\sum_{j=2}^{k}(-1)^j \frac{g_1^{k-j}g_j}{(k-j)!}
\end{equation}

Here are the explicit first $5$ terms of this formula:
$$
c_2=g_2-\frac{g_1^2}{2}
\quad
c_3=-g_3 + g_1 g_2 -\frac{g_1^3}{3}
\quad
c_4=g_4-g_1 g_3 + \frac{g_1^2 g_2}{2} - \frac{g_1^4}{8}
$$
$$
c_5=-g_5+g_1 g_4 - \frac{g_1^2 g_3}{2}+\frac{g_1^3 g_2}{6} -
\frac{g_1^5}{30} \quad c_6=g_6-g_1 g_5+\frac{g_1^2 g_4}{2} -
\frac{g_1^3 g_3}{6}+\frac{g_1^4 g_2}{24} - \frac{g_1^6}{144}.
$$

\bibliographystyle{apt}
\bibliography{biblio}

\end{document}